\documentclass[reqno]{amsart}%
\usepackage{amsmath}
\usepackage{amsfonts}
\usepackage{CJK}
\usepackage{amssymb}
\usepackage{graphicx}%
\usepackage{hyperref}

\newcommand{\cH}{\mathcal{H}}

\newtheorem{thm}{Theorem}[section]
\newtheorem{cor}[thm]{Corollary}

\newtheorem{prop}[thm]{Proposition}

\theoremstyle{definition}
\newtheorem{defn}{Definition}[section]
\theoremstyle{Conjecture}

\theoremstyle{remark}
\newtheorem{rem}{Remark}[section]
\theoremstyle{Example}
\newtheorem{exm}{Example}[section]

\newcommand{\be}{\begin{equation}}
\newcommand{\ee}{\end{equation}}
\newcommand{\bea}{\begin{eqnarray}}
\newcommand{\eea}{\end{eqnarray}}
\newcommand{\ben}{\begin{eqnarray*}}
\newcommand{\een}{\end{eqnarray*}}
\newcommand{\bet}{\begin{equation}
\begin{split}}
\newcommand{\eet}{\end{split}
\end{equation}}

\begin{document}
\title[Properties of squeezing functions]{Properties of squeezing functions and global transformations of bounded domains}
\date{}
\author[F. Deng]{Fusheng Deng}
\address{F. Deng: School of Mathematical Sciences, University of
Chinese Academy of Sciences, Beijing 100049, China}
\email{fshdeng@ucas.ac.cn}
\author[Q. Guan]{Qi'an Guan}
\address{Q. Guan: Beijing International Center for Mathematical Research, Peking University, Beijing, 100871, China}
\email{guanqian@amss.ac.cn}
\author[L. Zhang]{Liyou Zhang}
\address{L. Zhang: School of Mathematical Sciences, Capital Normal University, Beijing 100048, China}
\email{zhangly@mail.cnu.edu.cn}
\begin{abstract}
The central purpose of the present paper is to study boundary behavior
of squeezing functions on bounded domains.
We prove that the squeezing function of a strongly pseudoconvex
domain tends to 1 near the boundary.
In fact, such an estimate is proved for the squeezing function
on any domain near its globally strongly convex boundary points.
We also study the stability of squeezing functions on a sequence of bounded domains,
and give comparisons of intrinsic measures and metrics
on bounded domains in terms of squeezing functions.
As applications,
we give new and simple proofs of several well known
results about geometry of strongly pseudoconvex domains,
and prove that all Cartan-Hartogs domains are homogenous regular.
Finally, some related problems that ask for further study
are proposed.
\end{abstract}
\maketitle

\section{introduction}\label{sec:introduction}
In a recent work \cite{DGZ},
the authors introduced the notion of
squeezing functions to study geometric
and analytic properties of bounded domains.
The squeezing function of a bounded domain $D$
is defined as follows:
\begin{defn}\label{def:squeezing function}
Let $D$ be a bounded domain in $\mathbb{C}^n$.
For $z\in D$ and an (open) holomorphic embedding
$f: D\rightarrow B^n$ with $f(z)=0$,
we define
$$s_D(z , f)= \sup\{r|B^n(0,r)\subset f(D) \},$$
and the squeezing number $s_D(z)$ of $D$ at $z$
is defined as
$$s_D(z)= \sup_f\{s_D(z , f) \},$$
where the supremum is taken over all holomorphic
embeddings $f: D\rightarrow B^n$ with $f(z)=0$,
$B^n$ is the  unit ball in $\mathbb{C}^n$,
and $B^n(0 , r)$ is the ball in $\mathbb{C}^n$
with center $0$ and radius $r$.
As $z$ varies, we get a function $s_D$ on $D$,
which is called the \emph{squeezing function} of $D$.
\end{defn}

Roughly speaking,
$s_D(z)$ describes how does the domain $D$ look like the unit ball,
observed at the point $z$.
By definition, it is clear that squeezing functions are
invariant under biholomorphic transformations.
Namely,
if $f:D_1\rightarrow D_2$ is a holomorphic equivalence
of two bounded domains, then $s_{D_2}\circ f=s_{D_1}$.
Though the definition is so simple,
it is turned out that so many geometric and analytic properties
of  bounded domains are encoded in their squeezing functions.

Some interesting properties of squeezing functions
were established in \cite{DGZ}.
For example,
squeezing functions are continuous;
and for each $p\in D$,
there exists an extremal map realizing the supremum
in Definition \ref{def:squeezing function}.

In the present paper,
we continue to study squeezing functions and
their applications to geometry of bounded domains.
We first consider the stability of squeezing functions
on a sequence of domains.
We prove that,
for a sequence of increasing domains
convergent to a bounded domain,
the squeezing functions of these domains
converge to the squeezing function
of the limit domain.
We also prove a weaker result for
a sequence of decreasing domains.

A \emph{homogenous regular domain} (introduced in \cite{Liu})
is a bounded domain whose squeezing function
is bounded below by a positive constant.
By the famous Bers embedding (\cite{Bers}),
Teichm\"uller spaces of compact Riemann surfaces
are homogenous regular.
In the past decade,
comparisons of various intrinsic metrics
on Teichm\"uller spaces were extensively
studied (see e.g. \cite{Chen}\cite{Liu}\cite{Liu2}\cite{Yeung1}).
The equivalence of certain intrinsic measures
on Teichm\"uller spaces was proved in \cite{Overholser}.
In \cite{Liu},
it was proved that the Bergman metric,
the Kobayashi metric, and the Carath\'eodory metric
on a homogenous regular domain are equivalent.
Geometric and analytic properties of homogenous
regular domains were systematically studied in \cite{Yeung},
where the term homogenous regular domain was phrased
as uniformly squeezing domain.
In this paper,
we modify the methods in \cite{Overholser} and \cite{Yeung} to
give comparisons of intrinsic measures
and metrics on general bounded domains
in terms of squeezing functions.

The central problem in the theory is
studying boundary behavior of squeezing functions.
For a smoothly bounded planar domain $D$,
it was shown in \cite{DGZ} that
$\lim_{z\rightarrow \partial D}s_D(z)=1$.
In this paper, we try to generalize
this result to strongly pseudoconvex domains of higher dimensions.
To state the main results,
we first introduce the notion
of globally strongly convex boundary points.
Let $D$ be a bounded domain in $\mathbb{C}^n$
and $p\in\partial D$.
We call $p$ a \emph{globally strongly convex} (g.s.c)
boundary point of $D$ if $\partial D$ is $C^2$-smooth
and strongly convex at $p$, and $\bar D\cap T_p\partial D=\{p\}$,
where $T_p\partial D$ is the tangent hyperplane of
$\partial D$ at $p$.
In the present paper,
we will prove the following result:

\begin{thm}\label{thm:g.s.c}
Let $D\subset\mathbb{C}^n$ be a bounded domain.
Assume $p\in\partial D$ is a g.s.c
boundary point of $D$.
Then $\lim_{z\rightarrow p}s_D(z)=1.$
\end{thm}

Let $D$ be a bounded strongly pseudoconvex domain
and $p\in\partial D$.
It is well known that there is a local coordinate
change near $p$ such that $\partial D$ is strongly
convex at $p$ under the new coordinates.
It is natural to ask whether the coordinate change
can be taken globally on a neighborhood of $\bar D$?
A stronger version of this question is whether we can find a coordinate
change on a neighborhood $\bar D$ such that,
under the new coordinates,
$p$ is a g.s.c boundary point of $D$.
In a recent work,
Forn{\ae}ss and Wold answered this question affirmatively
and proved the following deep result:

\begin{thm}[\cite{Fornaess-Wold}]\label{thm:h.g.s.c}
Let $D$ be a bounded  strongly pseudoconvex
domain in $\mathbb{C}^n$ with $C^2$-smooth boundary
and $p\in\partial D$.
Then there is a neighborhood $\tilde D$ of $\bar D$ and
a holomorphic (open) embedding $f:\tilde D\rightarrow \mathbb{C}^n$
such that $f(p)$ is a g.s.c boundary point
of $f(D)$.
\end{thm}

Applying Theorem \ref{thm:g.s.c} and Theorem \ref{thm:h.g.s.c},
we can prove the following:

\begin{thm}\label{thm:s.p.c}
Let $D$ be a bounded  strongly pseudoconvex
domain in $\mathbb{C}^n$ with $C^2$-smooth boundary.
Then $\lim_{z\rightarrow \partial D}s_D(z)=1$.
In particular, by the continuity of squeezing functions,
$D$ is homogeneous regular.
\end{thm}

Theorem \ref{thm:s.p.c} provides a different approach
to the famous Levi's problem,
which states that all pseudoconvex domains are holomorphically convex.
We now explain the idea.
Let $D$ be a bounded strongly pseudoconvex domain.
By Theorem \ref{thm:s.p.c}, $D$ is homogenous regular.
As shown in \cite{DGZ}, it is then easy to prove that
the Carath\'eodory metric on $D$ is complete,
which implies that $D$ is a domain of holomorphy.
Note that any pseudoconvex domain can be approximated
by an increasing sequence of bounded strongly pseudoconvex domains,
and that the limit of an increasing sequence of holomorphically
convex domains is also holomorphically convex,
we get a solution to Levi's problem.

Based on Theorem \ref{thm:s.p.c} and other results in the present paper,
we can give a different proof of
Wong's result which says that
a bounded strongly pseudoconvex domain
is holomorphic equivalent to the unit ball
if its automorphism group is noncompact,
and give new and simple proofs of some other well known
results about geometry of strongly pseudoconvex domains
(see \S \ref{subsec:geo. s.p.c domains}).

%

Theorem \ref{thm:g.s.c} can also be applied to investigate the geometry
of Cartan-Hartogs domains, which are certain Hartogs domains with
classical bounded symmetric domains as bases.
Those domains have been extensively studied in the past decade
(see e.g. \cite{Yin2000,Englis2000, Wang2004, Roos2005, Loi, Park2012, Yamamori2012}).
Motivated by the works in \cite{Liu},
Yin proposed a problem that whether all Cartan-Hartogs domains are homogenous regular (\cite{Yin2007}).
In this paper, we answer this question affirmatively.
Consequently, all Cartan-Hartogs domains are hyperconvex and have bounded geometry; various
classical intrinsic metrics, as well as all the measures considered in \S \ref{sec:comparison. intrinsic form}, on these domains are equivalent.

Though squeezing functions are originally defined for bounded domains,
similar idea is possibly applied to algebraic geometry.
In fact, by Griffiths' results on uniformization  of algebraic varieties \cite{Griffiths71},
we can define squeezing functions on projective manifolds,
which will be studied more carefully elsewhere (see \S \ref{sec:further study}).

The rest of the paper is organized as follows.
In \S \ref{sec:squ. func. limit domain}, we study the stability of squeezing functions on a sequence of domains.
In \S \ref{sec:comparison. intrinsic form}, we give comparisons of intrinsic measures and metrics in terms of squeezing functions.
We prove Theorem \ref{thm:g.s.c} and Theorem \ref{thm:s.p.c} in \S \ref{sec:bound estimate}.
In \S \ref{sec:application}, we apply the results in previous sections to study geometry of Cartan-Hartogs domains and strongly pseudoconvex domains.
In the final \S \ref{sec:further study},
we propose some related problems that ask for further studies.

\vspace{.1in} {\em Acknowledgements}. The authors are grateful to J. E. Forn{\ae}ss
for showing them the preprint \cite{Fornaess-Wold}.
They like to thank B.Y. Chen,  K.-T. Kim,
K.F. Liu, P. Pflug, S.-K. Yeung, W.P. Yin, and X.Y. Zhou for helpful discussions.
The authors are partially supported by NSFC grants (10901152 and 11001148),
BNSF(No.1122010) and the President Fund of GUCAS.

\section{Stability of squeezing functions}\label{sec:squ. func. limit domain}

In this section, we consider the relation between the limit of squeezing functions of a sequence of
domains and the squeezing function of the limit domain.
For a sequence of increasing domains, we have the following

\begin{thm}\label{thm: increase limit domain}
Let $D\subset\mathbb{C}^n$ be a bounded domain and $D_k\subset D$ $(k\in\mathbb{N})$
be a sequence of domains such that $\cup_kD_k=D$ and $D_k\subset D_{k+1}$ for all $k$.
Then, for any $z\in D$,  $\lim_{k\rightarrow\infty}s_{D_k}(z)=s_D(z).$
\end{thm}
\begin{proof}
By the existence of extremal maps w.r.t squeezing functions (see Theorem 2.1 in \cite{DGZ}), for each $k$, there is an injective holomorphic map $f_k: D_k\rightarrow B^n$ such that $f_k(z)=0$ and $B^n(0,s_{D_k}(z))\subset f_k(D_k)$. By Montel's theorem, we may assume the sequence $f_k$ converges uniformly on compact subsets of $D$ to a holomorphic map $f: D\rightarrow \mathbb{C}^n$.

We first prove that $f$ is injective. Assume $z\in D_{k_0}$ for some $k_0>0$, then it is clear that $$s_{D_k}(z)\geq \frac{d(z,\partial D_k)}{diam(D_k)}\geq \frac{d(z,\partial D_{k_0})}{diam(D)}$$
 for $k>k_0$. So there is a $\delta>0$ such that $B^n(0,\delta)\subset f_k(D_k)$ for all $k>k_0$. Set $g_k = f^{-1}_k|_{B^n(0,\delta)}:B^n(0,\delta)\rightarrow D$. By Cauchy's inequality, $|\det(dg_k(0))|$ is bounded above  uniformly  for all $k>k_0$ by a positive constant. Hence there exits a constant $c>0$, such that $|\det(df_k(z))|>c$ for all $k>k_0$. This implies $\det(df(z))\neq 0$. So the injectivity of $f$ follows from Lemma 2.3 in \cite{DGZ} and the generalized Rouch\'e's theorem (Theorem 3 in \cite{Lloyd}).

 Since $f$ is injective, it is an open map (see e.g. Theorem 8.5 in \cite{Fritzsche}). On the other hand, it is clear that $f(D)\subset \overline{B^n}$. So we have $f(D)\subset B^n$.

We now prove that $s_D(z)\geq\limsup_k s_{D_k}(z)$. Let $s_{D_{k_i}}$ be a subsequence such that
 $\lim_{k_i\rightarrow \infty}s_{D_{k_i}}(z)=\limsup_ks_{D_k}(z)=r$, then, as explained above, we have $r>0$. Let $\epsilon>0$ be an arbitrary positive number less than $r$, then $B^n(0,r-\epsilon)\subset f_k(D_{k_i})$ for $k_i$ large enough. Set $h_{k_i}=f^{-1}_{k_i}|_{B^n(0,r-\epsilon)}$, then $\lim_{k_i\rightarrow\infty}|\det(dh_{k_i}(z))|=|\det(df^{-1}(0))|\neq 0$. Bcy the argument mentioned above, $h:=\lim_{k_i}h_{k_i}$ is injective and hence $h(B^n(0,r-\epsilon))\subset D$. This implies $f(h(w))$ make sense for all $w\in B^n(0,r-\epsilon)$. It is clear that $f(h(w))=w$ for all $w\in B^n(0,r-\epsilon)$. So $B^n(0,r-\epsilon)\subset f(D)$ and $s_D(z)\geq r-\epsilon$. Since $\epsilon$ is arbitrary, we get $s_D(z)\geq\limsup_k s_{D_k}(z)$.

 Finally, we prove that $s_D(z)\leq \liminf_ks_{D_k}(z)$. Let $s_{D_{k'_i}}$ be a subsequence such that
 $\lim_{k'_i\rightarrow \infty}s_{D_{k'_i}}(z)=\liminf_ks_{D_k}(z)$. By the existence of extremal map, there exists an injective holomorphic map
 $\varphi:D\rightarrow B^n$ such that $\varphi(z)=0$ and $B^n(0,s_D(z))\subset \varphi(D)$. For arbitrary $0<\epsilon<s_D(z)$, by assumption, $\varphi^{-1}(B^n(0,s_D(z)-\epsilon))\subset D_{k'_i}$ for $k'_i$ large enough. So, for $k'_i$ large enough, we have $s_{D_{k'_i}}(z)\geq s_D(z)-\epsilon$. This implies $s_D(z)-\epsilon\leq \lim_{k'_i\rightarrow \infty}s_{D_{k'_i}}(z)$. Since $\epsilon$ is arbitrary, we get
 $s_D(z)\leq \lim_{k'_i\rightarrow \infty}s_{D_{k'_i}}(z)=\liminf_ks_{D_k}(z)$.
\end{proof}


For a sequence of decreasing domains, we have
\begin{thm}\label{thm: decrease limit domain}
Let  $D\subset\mathbb{C}^n$ be a bounded domain and $D_k\supset D$ $(k\in\mathbb{N})$ be a sequence of domains such that $\cap_kD_k=D$ and $D_{k+1}\subset D_{k}$ for all $k$. Then, for any $z\in D$,  $s_D(z)\geq \limsup_ks_{D_k}(z)$.
\end{thm}
\begin{proof}
For each $k$, let $f_k: D_k\rightarrow B^n$ an injective holomorphic map such that $f_k(z)=0$ and $B^n(0,s_{D_k}(z))\subset f_k(D_k)$. By Montel's theorem, we may assume $\lim_k f_k=f$ exists and give a holomorphic map from $D$ to $\mathbb{C}^n$. By the same argument as in proof of Theorem \ref{thm: increase limit domain}, we see that
$f$ is injective and $f(D)\subset B^n$.

Without loss of generality, we assume $\lim_ks_{D_k}(z)=r$. Then, for any $\epsilon>0$, $B^n(0,r-\epsilon)\subset f_k(D_k)$ for $k$ large enough. Set $g_k=f_k^{-1}|_{B^n(0,r-\epsilon)}:B^n(0,r-\epsilon)\rightarrow D_k$. We can assume $g_k$ converges uniformly on compact subsets of $B^n(0,r-\epsilon)$ to a holomorphic map $g:B^n(0,r-\epsilon)\rightarrow \mathbb{C}^n$. Similarly, one can show that $g$ is injective and hence open. On the other hand, by assumption, it is clear that $g(B^n(0,r-\epsilon))\subset \cap_{k\geq 1}\overline{D_k}$. Hence $g(B^n(0,r-\epsilon))\subset \cap_{k\geq 1}D_k=D$. This implies that $f(g(w))$ makes sense for all $w\in B^n(0,r-\epsilon)$. It is clear that $f(g(w))=w$ for all $w\in B^n(0,r-\epsilon)$. So $B^n(0,r-\epsilon)\subset f(D)$ and $s_D(z)\geq r-\epsilon$. Since $\epsilon$ is arbitrary, we get $s_D(z)\geq r=\lim_k s_{D_k}(z)$.

\end{proof}

The following example shows that the strict inequality in Theorem \ref{thm: decrease limit domain} is possible:
\begin{exm}
Let $D=\{(z_1,z_2)| 0<|z_2|<|z_1|<1\}$ be the Hartogs triangle in $\mathbb{C}^2$. For a positive number $\epsilon$ (small enough), we define a domain $V_\epsilon$ in $\mathbb{C}^2$ as
$$V_\epsilon = \{(z_1,z_2)|0<|z_1|<1, 0<|z_2|<\epsilon\}.$$
Set $D_\epsilon = D\cup V_\epsilon$. Let $z^j=(z^j_1, z^j_2)$ be a sequence of points in $D$ satisfying the conditions $|z_1^j|\leq (1+\frac{1}{j})|z^j_2|$ and $|z_2^j|>a$ for all $j$, where $a>0$ is a fixed constant. Then we have\\
1). $\lim_{j\rightarrow\infty}s_{D_\epsilon}(z^j)=0$ uniformly with respect to $\epsilon$, and\\
2). there exists a positive constant $c$, such that $s_D(z^j)\geq c$ for all $j$.
\end{exm}
\begin{proof}
1) By the Riemann's removable singularity theorem and Hartogs's extension theorem, the Carath\'eodory metric $C_{D_\epsilon}$ on $D_\epsilon$ is given by the restriction on $D_\epsilon$ of the Carath\'eodory metric on $\Delta\times\Delta$. Note that the Carath\'eodory metric on $\Delta\times\Delta$ is continuous, it is clear that there exists a sequence of positive numbers $r^j$ such that $\lim_{j}r_j=0$ and the balls, denoted by $B_\epsilon(z^j, r^j)$, in $D_\epsilon$ centered at $z^j$ with radius $r^j$ with respect to $C_{D_\epsilon}$,  are not relatively compact in $D_\epsilon$ for all $j$ and all $\epsilon$ small enough. Assume $f:D_\epsilon\rightarrow B^2$ is an injective holomorphic map such that $f(z^j)=0$ and $B^2(0,s_{D_\epsilon}(z^j))\subset f(D_\epsilon)$. By the decreasing property of Carath\'eodory metric, we see that $f(B_\epsilon(z^j, \sigma(\frac{s_{D_\epsilon}(z^j)}{2})))$ is relatively compact in $f(D_\epsilon)$, where $\sigma:[0,1)\rightarrow \mathbb{R}$ is the function defined as $\sigma(x)=\ln\frac{1+x}{1-x}$. Since $f$ is injective, this implies $s_{D_\epsilon}(z^j)\leq 2\sigma^{-1}(r^j)$ for all $j$. Hence $\lim_{j\rightarrow\infty}s_{D_\epsilon}(z^j)$ tends to 0 uniformly w.r.t $\epsilon$.\\
2)The map $\varphi(z_1,z_2)=(z_1, \frac{z_2}{z_1})$ gives a holomorphic isomorphism from $D$ to $\Delta^*\times\Delta^*$. Denote $\varphi(z^j)$ by $(w^j_1,w^j_2)$, then $|w^j_1|, |w^j_2|>a$. Note that the squeezing function on $\Delta^*$ is given by $s_{\Delta^*}(z)=|z|$ (see Corollary 7.2 in \cite{DGZ}), so we have $s_{\Delta^*\times\Delta^*}(w^j_1,w^j_2)\geq \frac{\sqrt{2}}{2}a$ for all $j$. By the holomorphic invariance of squeezing functions, we get $s_D(z^j)\geq \frac{\sqrt{2}}{2}a$ for all $j$.
\end{proof}

\section{Comparisons of intrinsic measures and metrics}\label{sec:comparison. intrinsic form}
In this section,
we give comparisons of intrinsic measures
and metrics on bounded domains
in terms of squeezing functions.

\subsection{Comparisons of intrinsic measures}

Let $D$ be a domain in $\mathbb{C}^n$.
The \emph{Carath\'{e}odory measure} on $D$ is defined to be the $(n,n)$-from
$$\mathcal{M}^C_D(z)= M^C_D(z)\frac{i}{2}dz_1\wedge d\bar{z}_1\wedge\cdots\wedge\frac{i}{2}dz_n\wedge d\bar{z}_n,$$
where
$$M^C_D(z)=\sup\{|\det f'(z)|^2; f:D\rightarrow B^n\ \text{\ holomorphic\ with}\ f(z)=0\};$$
and the \emph{Eisenman-Kobayashi measure } on $D$ is defined to be the $(n,n)$-from
$${\mathcal{M}}^K_D(z)= M^K_D(z)\frac{i}{2}dz_1\wedge d\bar{z}_1\wedge\cdots\wedge\frac{i}{2}dz_n\wedge d\bar{z}_n,$$
where
$$M^K_D(z)=\inf\{1/|\det f'(0)|^2; f:B^n\rightarrow D\ \text{\ holomorphic\ with}\ f(0)=z\}.$$

The Carath\'eodory measure and the
Eisenman-Kobayashi measure satisfy
the decreasing property.
Namely, if $f:D_1\rightarrow D_2$ be a holomorphic map
between two domains in $\mathbb{C}^n$,
then $f^{*}\mathcal M^C_{D_2}\leq \mathcal M^C_{D_1}$
and $f^{*}\mathcal M^K_{D_2}\leq \mathcal M^K_{D_1}$.

Let $h$ be a norm on $\mathbb{C}^n$,
and let $B^n(h):=\{v\in\mathbb{C}^n| h(v)<1 \}$
be the unit ball with respect to $h$.
Then the measure of $h$ is defined as
$$\frac{vol(B^n)}{vol(B^n(h))}\frac{i}{2}dz_1\wedge d\bar{z}_1\wedge\cdots\wedge\frac{i}{2}dz_n\wedge d\bar{z}_n,$$
where $vol(B^n)$ and $vol(B^n(h))$ denote
the Euclidean volumes of $B^n$ and $B^n(h)$ respectively.
Note that the measure of $h$ is completely determined by $h$,
and independent of the choice the original
inner product on $\mathbb{C}^n$.

On a bounded domain $D$,
the Kobayashi metric and the Carath\'eodory metric (see e.g. \cite{Pflug} for an introduction) are nondegenerate,
namely, they give norms on tangent spaces at all points of $D$.
So we can define the measures of the Kobayashi metric
and the Carath\'eodory metric on $D$ and denote them
by $\tilde{\mathcal{M}}^K_D$ and $\tilde{\mathcal{M}}^C_D$ respectively.
Since  the Kobayashi metric and the Carath\'edory metric
satisfy the decreasing property (see e.g. \cite{Pflug}),
so do their measures.
Here one should note that,
in general,
the Carath\'eodory (resp. Eisenman-Kobayashi) measure and  the measure
of the Carath\'eodory (resp. Kobayashi) metric are different.

On the unit ball $B^n$, all the four measures defined above coincide.
Let $\mathcal{M}$ and $\mathcal{M}'$ be any two of the four measures, i.e.,
the Carath\'eodory measure, the Eisenman-Kobayashi measure,
the measure of the Carath\'eodory metric,
and the measure of the Kobayashi metric.
Then we have the following:

\begin{thm}\label{thm:Comp. of Car. Kob. Vol.}
Let $D$ be a bounded domain in $\mathbb{C}^n$,  $\mathcal{M}$ and $\mathcal{M}'$ as above.
Then we have
$$s^{2n}_D(z)\mathcal{M}'_D(z)\leq\mathcal{M}_D(z)\leq \frac{1}{s_D^{2n}(z)}\mathcal{M}'_D(z), z\in D.$$
In particular, if $D$ is  homogenous regular and $s_D(z)\geq c>0$, then
$$c^{2n}\mathcal{M}'_D(z)\leq\mathcal{M}_D(z)\leq \frac{1}{c^{2n}}\mathcal{M}'_D(z), z\in D,$$
and hence $\mathcal{M}_D$ and $\mathcal{M'}_D$ are equivalent.
\end{thm}
\begin{proof}
Let $f:D\rightarrow B^n$ be a holomorphic injective map with $f(z)=0$
and $B^n(0, s_D(z))\subset f(D)$.
Note that $\mathcal M_{B^n}=\mathcal M'_{B^n}=r^{2n}\mathcal M_{B^n(0,r)}.$
By the decreasing property of $\mathcal M$ and $\mathcal M'$,
we have
$$s_D^{-2n}(z)\mathcal M_{B^n}(0)\geq\mathcal M_{f(D)}(0)\geq \mathcal M_{B^n}(0),$$
$$s_D^{-2n}(z)\mathcal M_{B^n}(0)\geq\mathcal M'_{f(D)}(0)\geq \mathcal M_{B^n}(0).$$
Note that $\frac{\mathcal M_{f(D)}(0)}{\mathcal M'_{f(D)}(0)}=\frac{\mathcal M_{D}(0)}{\mathcal M'_{D}(0)}$,
we get
$$s^{2n}_D(z)\mathcal{M}'_D(z)\leq\mathcal{M}_D(z)\leq \frac{1}{s_D^{2n}(z)}\mathcal{M}'_D(z), z\in D.$$
\end{proof}

\subsection{Comparisons of intrinsic  metrics}
It is known that the Kobayashi metric
and Carath\'eodory metric
on bounded domains are Finsler metrics
satisfying the decreasing property.
They are coincide on the unit ball.
It is also well known that the Carath\'eodory metric
on a bounded domain is dominated by its Kobayashi metric.
Let $D$ be a bounded domain, and denote by
$\mathcal H_D^K$ and $\mathcal{H}_D^C$
the Carath\'eodory metric and the Kobayashi metric on
$D$ respectively.
With the same argument as in the proof of Theorem \ref{thm:Comp. of Car. Kob. Vol.},
one can prove the following

\begin{thm}\label{thm:comp. Kob. and Car. metric}
Let $D$ be a bounded domain in $\mathbb{C}^n$.
Then
$$s_D(z)\mathcal{H}^K_D(z)\leq\mathcal{H}^C_D(z)\leq \mathcal{H}^K_D(z).$$
In particular, if $D$ is homogenous regular and $s_D(z)\geq c>0$,
then, for any $z\in D$,
$$c\mathcal{H}^K_D(z)\leq\mathcal{H}^C_D(z)\leq \mathcal{H}^C_K(z),$$
and hence $\mathcal{H}^C_D$ and $\mathcal{H}^K_D$ are equivalent.
\end{thm}

For a bounded domain $D$,
we have got a comparison between its
Carath\'eodory metric $\cH_D^C$ and Kobayashi metric
$\cH_D^K$ in terms of its squeezing function
in Theorem \ref{thm:comp. Kob. and Car. metric}.
The Bergman metric $\cH^B_D$,
which does not satisfy the decreasing property,
is invariant under biholomorphic transformations.
When $D$ is pseudoconvex,
it is well known that there is a unique
complete K\"ahler-Einstein metric on $D$,
denoted by $\cH^{KE}_D$,
with Ricci curvature normalized by $-(n+1)$ \cite{Mok-Yau},
which is also invariant under biholomorphic transformations.
\begin{thm}\label{thm:comp. Beg. KE metrics}
Let $D$ be a bounded domain in $\mathbb{C}^n$ and $z\in D$,
and let $s_D$ be the squeezing function on $D$.
Then
\be\label{eqn:comp. Kob. and Berg.}
s_D(z)\cH_D^K(z)\leqslant \cH_D^B(z) \leqslant \frac{2^{n+2}\pi}{s^{n+1}_D(z)}\cH_D^K(z).
\ee
If in addition $D$ is pseudoconvex, then
\be\label{eqn:comp. Kob. and KE.}
\sqrt{\frac{1}{n}}s_D(z) \cH_D^K(z)\leqslant \cH_D^{KE}(z)\leqslant\left(\frac {n}{s^2_D(z)}\right)^{(n-1)/2}\cH_D^K(z).
\ee
\end{thm}

\begin{rem}
If $D$ is homogenous regular and $s_D(z)\geq c$
for some constant $c>0$,
the above comparison,
with $s_D(z)$ replaced by $c$,
was proved in \cite{Yeung}.
In particular,
the Bergman metric and the K\"ahler-Einstein metric
on $D$ are equivalent to the Kobayashi metric.
A slight modification of the method in \cite{Yeung}
can be used to give the proof of
Theorem \ref{thm:comp. Beg. KE metrics},
so we omit it here.
\end{rem}

For a metric $h$ on a bounded domain $D$, as explained in the above subsection,
we can define the measure $\mathcal{M}^h$ of $h$.
If there are two  metrics $h$ and $h'$ on $D$ satisfying the condition
$$a(z)h'(z)\leq h(z)\leq b(z)h'(z), z\in D,$$
where $a$ and $b$ are two continuous strictly positive functions on $D$,
then the measures $\mathcal{M}^h$ and $\mathcal{M}^{h'}$ satisfy the comparison
$$(a(z))^{2n}\mathcal{M}^{h'}(z)\leq \mathcal{M}^h(z)\leq (b(z))^{2n}\mathcal{M}^{h'}(z), z\in D.$$
In particular, if $h$ and $h'$ are equivalent,
then $\mathcal{M}^h$ and $\mathcal{M}^{h'}$ are also equivalent.

We have shown in Theorem \ref{thm:Comp. of Car. Kob. Vol.} that the measures of the Kobayashi metric
and the Carath\'eodory metric on a homogenous regular domain are equivalent,
and they are equivalent to the Carath\'eodory measure and the Kobayashi measure.
We also see that, on a homogenous regular domain, the Kobayashi metric
and the Carath\'eodory metric are equivalent.
By Theorem \ref{thm:comp. Beg. KE metrics},
they are equivalent to the Bergman metric
and the K\"ahler-Einstein metric.
As a consequence, we have
\begin{thm}
On a homogenous regular domain,
the measures of the Kobayashi metric,
the Carath\'eodory metric,
the Bergman metric,
and the K\"ahler-Einstein metric are equivalent,
and they are equivalent to the Carath\'eodory
and the Eisenman-Kobayashi measures.
\end{thm}

The equivalence of some of the above measures
was established in \cite{Overholser} for
Teichm\"{u}ller spaces,
which is one kind of homogenous regular domains.

%

\section{Boundary behavior of squeezing functions}\label{sec:bound estimate}
Let $D$ be a bounded domain and $p\in\partial D$.
Recall that $p$ is called a \emph{globally strongly convex} (g.s.c)
boundary point of $D$
if $\partial D$ is $C^2$-smooth
and strongly convex at $p$,
and $\bar D\cap T_p\partial D=\{p\}$,
where $T_p\partial D$ is the tangent hyperplane of
$\partial D$ at $p$.
In this section we will prove the following

\begin{thm}\label{thm:g.s.c text}
Let $D\subset\mathbb{C}^n$ be a bounded domain.
If $p$ is a g.s.c boundary point of $D$,
then $\lim_{z\rightarrow p}s_D(z)=1.$
\end{thm}

Note that squeezing functions are invariant under biholomorphic
transformations.
By Forn{\ae}ss and Wold's result (Theorem \ref{thm:h.g.s.c}),
we get the following

\begin{thm}\label{thm:h.g.s.c text}
Let $D\subset\mathbb C^n$ be a bounded strongly pseudoconvex
domain with $C^2$-smooth boundary.
Then $\lim_{z\rightarrow \partial D}s_D(z)=1$.
\end{thm}

With Theorem \ref{thm:h.g.s.c text},
we can a different proof of B. Wong's result
on characterization of the unit ball by its symmetry
from strongly pseudoconvex domains,
based on the original idea of localization of Wong.

To prove Theorem \ref{thm:g.s.c text},
we need to introduce a new function $e_D$ on $\partial D$.
Let $D$ be a bounded domain in $\mathbb{C}^n$
and $p\in\partial D$.
If $D$ is $C^2$-smoothly bounded at $p$ and
contained in some ball in $\mathbb{C}^n$ with boundary point $p$,
then $e_D(p)$ is defined to be the minimum of the radii of balls
with boundary point $p$ which contain $D$.
If $D$ is not $C^2$-smoothly bounded at $p$ or
no ball with boundary point $p$ can contain $D$,
we set $e_D(p)=+\infty$.
By definition,
it is clear that $p$ is a g.s.c boundary point of $D$
if and only if $e_D(p)<\infty$.
The following proposition is an important step
in our proof of Theorem \ref{thm:g.s.c text}.

\begin{prop}\label{prop:semi-conti. of B_D}
Let $D$ be a bounded domain.
Then $e_D$ is upper semi-continuous on $\partial D$.
\end{prop}
\begin{proof}

For $r>0$ and $q\in \partial D\cap U$, let $B_{q,r}$ be the ball defined by
$$|z-(q-r\nabla\rho(q))|^2<r^2.$$
Let $r>e_D(p)$ be fixed,  we want to prove that, for some neighborhood $V\subset U$ of $p$, $D\subset B_{q,r}$ for all $q\in \partial D\cap V$.
Let
$$f_r(z,q)=\frac{|z-(q-r\nabla\rho(q))|^2-r^2}{2r}.$$
By assumption, we can choose a local defining function $\rho$ of $D$ near $p$ such that
$||\nabla \rho||\equiv 1$ and $Hess(\rho)(p)> cHess(f_r(z,p))|_{z=p}$ for some $c>1$.
By continuity, there is a neighborhood $W$ of $p$ such that
\be\label{eqn:comp. of Hessen}
Hess(\rho)(q)> cHess(f_r(z,q))|_{z=q}
\ee
for $q\in \partial D\cap W$.
We may assume $W$ is convex and small enough. Then, for any fixed $q\in \partial D\cap W$,  we have
$$\rho(z)=\Delta x\cdot\nabla\rho(q)+\sum_{i,j=1}^{2n}h_{i,j}(z,q)\Delta x_i\Delta x_j,$$
where $\Delta x=(\Delta x_1,\cdots,\Delta x_{2n})=z-q$ is viewed as a vector in $\mathbb{R}^{2n}$.
The key point here is that all $h_{i,j}(z,q)$ are continuous on $W\times (W\cap\partial D)$,
and $h_{i,j}(q,q)=\frac{\partial^2\rho}{\partial x_i\partial x_j}(q).$
By \eqref{eqn:comp. of Hessen}, replacing $W$ by a small enough relatively open subset of it, we have
$$\rho(z)-f_r(z,q)=\sum_{i,j=1}^{2n}h_{i,j}(z,q)\Delta x_i\Delta x_j-\Delta x Hess(f_r(z,q))|_{z=q}\Delta x^T>0$$
for $(z,q)\in W\times (W\cap\partial D).$
This implies that $W\subset B_{q,r}$ for all $q\in \partial D\cap W$.

On the other hand, it is clear that there is an open subset $V$ of $W$ such that $D-W\subset B_{q,r}$ for all $q\in \partial D\cap V$.
So, for all $q\in \partial D\cap V$, we have $D\subset B_{q,r}$.
This implies $e_D(q)\leq r$. Let $r\searrow e_D(p)$,
we see that $e_D$ is upper semi-continuous at $p$.

\end{proof}

We now give the proof of Theorem \ref{thm:g.s.c text}:

\begin{proof}

By scaling if necessary, we can assume $e_D(p)\leq 1.$
Let
$$B=\{z\in\mathbb{C}^n: 2 \text{Re} z_1+\sum_{j=1}^n |z_i|^2<0\}$$
be the ball of radius 1 and centered at $(-1, 0, \cdots, 0)$.
It was shown in \cite{F-M} that there a series of biholomorphic transformations
that map $D$ to a domain, say, $D'$ in $B$ and map $p$ to the origin $0$,
such that $\bar{D'}\cap \partial B=0$ and
\begin{equation}\label{eqn:trans. Conveg.}
\lim_{r\rightarrow 0}s_{D'}(r, 0, \cdots, 0)=1.
\end{equation}
In other words, $s_{D'}(z)$ tends to $1$ as $z$ tends to $0$ from the normal direction
of $\partial D'$ at the origin.
We assume $D\subset B$ and $p$ is the origin.
Then the process of transformation is as follows.

\bigskip
\emph{Step 1.}
After a unitary transformation if necessary,
we can assume the defining function $\rho(z)$ of $D$ near $p=0$ can be written as
\begin{equation}\label{eqn:def. function}
\begin{split}
\rho(z)=&2Rez_1+Re\sum_{i,j=1}^n a_{ij}(p)z_iz_j\\
        &Re\sum_{j=1}^nc_j(p)z_1\bar z_j+\sum_{j=2}^n N_j(p)|z_j|^2+o(|z|^2).
\end{split}
\end{equation}
Let $H_{c,N}:\mathbb{C}^n\rightarrow\mathbb{C}^n$ be the biholomorphic map given by
$$H_1(z)=z_1,\ H_j(z)=z_j+(c_j(p)/(2N_j(p)))z_1, j=2,\cdots, n.$$
Shrinking $H_{c,N}(D)$ in all directions and rescaling it in the $z'$ directions,
we obtain a domain $D_1\subset B$ whose defining function near the origin is of the form
\begin{equation}
\begin{split}
\rho'(z)=&2Rez_1+2Re\sum_{j=1}^n b_jz_1z_j\\
        &+Re\sum_{i,j=1}^na_{ij}z_i z_j+d|z_1|^2+M|z'|^2+ o(|z|^2).
\end{split}
\end{equation}

\bigskip

\emph{Step 2.} For $\epsilon>0$, we define a biholomorphic map $f_\epsilon:B\rightarrow B$ as:
$$w_1=\frac{\epsilon z_1}{2-\epsilon+(1-\epsilon)z_1},\
w'=\frac{\sqrt{\epsilon(2-\epsilon)}}{2-\epsilon+(1-\epsilon)z_1}z'.$$
For $\epsilon$ small enough such that
$$b'_1=\frac{\epsilon}{2-\epsilon}b_1,\ b'_j=\sqrt{\frac{\epsilon}{2-\epsilon}}b_j$$
satisfy the condition
\begin{equation}
\sum_{j=1}^n|b_j|\leq\lambda,
\end{equation}
where $\lambda$ is a given uniform constant,
then the domain $\tilde{D}=G_{b'}(f^{-1}_\epsilon(D_1))$ has a defining function near 0 of the form
$$\rho''(z)=2Rez_1+Re\sum_{i,j=1}^n a_{ij}z_iz_j
       +d|z_1|^2+M|z'|^2+ o(|z|^2),$$
where the map $G_{b'}$ is defined as $G_{b'}(z)=(z_1+\sum_{j=1}^n b'_jz_1z_j, z')$.
We choose an integer $N> 8M$ and set $D_2=\frac{M}{N}\tilde{D}$.

\bigskip
\emph{Step 3.} Let $a_{ij}(p)$ as in Step 1, and set
$$b_{ij}=\frac{1}{4(N-4-k/2)}a_{ij}(p), 2\leq i, j\leq n,$$
where $k<2(N-4)$ is a constant.
We choose $\epsilon>0$ small enough and set $D_3=h\circ F_b \circ f^{-1}_\epsilon(D_2)$,
where $F_b:\mathbb{C}^n\rightarrow \mathbb{C}^n$ is given by
$$F_b(z_1, z')=\left(z_1+\sum_{i,j=2}^n b_{ij}z_iz_j,\ z' \right),$$
and $h:\mathbb{C}^n\rightarrow \mathbb{C}^n$ is
given by $h(z_1, z')= \left(\frac{14}{15}z_1, \sqrt{\frac{14}{15}z'}\right).$

Repeat Step 3 several times ($\leq 2(N-4)$) if necessary,
we get a domain, say $D'$, whose squeezing function satisfies the property
stated in the beginning of the proof.

\bigskip

Now we move forward to prove that $\lim_{z\rightarrow p}s_D(z)=1$.
Let $$l_p=\{(-r, 0')\in \bar D': 0\leq r<\delta_p\},$$
where $\delta_p>0$ very small.
We have seen that $s_{D'}(z)$ converges to 1 when $z$ tends to
the origin along $l_p$.

Let $l^-_p$ be the inverse image of $l_p$ in $D$
under the series of transformations given in the
above steps.
Then $l^-_p$ is a smooth curve in $\bar D$ through $p$.
Since all the transformations in the above steps
are defined in some neighborhood of the closure of the domains
involved,
it is clear that there is a constant $c_p>1$ such that
\begin{equation}\label{eqn:distance comparison}
c^{-1}_p\leq \frac{d(z,p)}{r(z)}\leq c_p
\end{equation}
for all $z\in l_p^-$,
where $d(z , p)$ is the length of the part of $l_p^-$
between $z$ and $p$,
and $(-r(z), 0')\in l_p$ is the image of $z$.

The speed of the convergence of \eqref{eqn:trans. Conveg.} depends on
the eigenvalues  $c_j(p)$ and $N_j(p)$ of the local defining function
that given in \eqref{eqn:def. function}.
However, a direct computation shows that
the convergence is uniform if $c_j$ and $N_j$
lie in a bounded set.

We can do the same process for other boundary points near $p$.
Note that $e_D(p)<1$, by Proposition \ref{prop:semi-conti. of B_D},
there is a neighborhood $U$ of $p$ in $\partial D$
such that $e_D(q)<1$ for all $q\in U$.
Repeat the above process,
we get a smooth curve $l^-_q$ of length $\geq \delta_q/c_q$ in $\bar D$ though $q$
such that $s_D(z)$ tends to $1$ as $z$ tends to $q$ along $l^-_q.$

We need to check how the above Steps depend on $p$.
In Step 1, we first meet the map $H$, which smoothly
depends on parameters $c_j(p)$ and $N_j(p)$.
By the Gram-Schmidt process,
we see that $c_j(q)$ and $N_j(q)$
can smoothly vary with respect to $q\in U$.
Moreover this implies that
the shrinking and rescaling appeared
in Step 1 can be taken to be uniform for $q\in U$.
This also implies that the positive numbers $\epsilon$
appearing in Step 2 and Step 3 can be taken independent of $q\in U$.
So, it is clear that all other transformations appearing
in Step 2 and Step 3 are also smoothly depend on $q\in U$.
Moreover,
for $q\in U$, $c_j(q)$ and $N_j(q)$ vary in a compact set
in $\mathbb{C}$.
One can also see that the first and second derivatives of
these transformations are continuous,
so $\delta_q>0$ and $c_q>0$
can be taken independent of $q$.

We choose $\delta>0$ and define a map
$\varphi:[0,\delta)\times U\rightarrow \bar D$ such that
$\varphi(t, q)$ is the unique point in $l_q^{-}$ with $d(\varphi(t, q), q)=t$.
By the above discussion, $\varphi$ is a smooth map.
It is clear that the tangent vector of $l^-_p$ at $p$
is not tangent to $\partial D$.
So the differential of $\varphi$ at $p$ is a linear isomorphism.
By the inverse function theorem,
$\varphi$ is a local diffeomorphism near $p$.
Without loss of generality,
we may assume $\varphi:[0,\delta)\times U\rightarrow \varphi([0,\delta)\times U)$ is a diffeomorphism.
Hence, for each $z\in \varphi((0,\delta)\times U)$,
there is a unique $q_z\in U$ such that $z\in l^-_{q_z}$,
and $d(z , q_z)$ tends to 0 uniformly
as $z\rightarrow \partial D$ uniformly.
By the above discussion, we see $\lim_{z\rightarrow p}s_D(z)=1$.
\end{proof}

\begin{cor}[\cite{Wong}]\label{cor:Wong ball}
Let $D$ a bounded strongly pseudoconvex
domain in $\mathbb{C}^n$ with $C^2$-smooth boundary.
If the automorphism group $Aut(D)$
of $D$ is noncompact,
then $D$ is biholomorphic to the unit ball.
\end{cor}
\begin{proof}
If $Aut(D)$ is noncompact,
then, for any $z\in D$, there is a sequence $f_j\in Aut(D)$ and a point $p\in\partial D$
such that $\lim_{j\rightarrow \infty}f_j(z)\rightarrow p$.
By the holomorphic invariance of squeezing functions,
$s_D(z)=s_D(f_j(z))$ for all $j$.
By Theorem \ref{thm:h.g.s.c text},
$\lim_{j\rightarrow \infty}s_D(f_j(z))=1$.
Hence $s_D(z)=1$.
By Theorem 2.1 in \cite{DGZ},
$D$ is biholomorphic to the unit ball.
\end{proof}

\section{Applications}\label{sec:application}
\subsection{Geometry of Cartan-Hartogs domains}\label{sec:Cartan-Hartogs}
In this subsection,
we apply Theorem \ref{thm:g.s.c} to investigate squeezing functions and
geometry of Cartan-Hartogs domains,
which are certain Hartogs domains based on
classical bounded symmetric domains.

Recall that a classical bounded symmetric domain
is a domain of one of the following four types:
\begin{equation*}
\begin{split}
& D_I(r,s)=\{Z=(z_{jk}): I-Z\bar{Z^t}>0,\ \text{where}\ Z\ \text{is\ an}\ r\times s \ \text{matrix} \}\ (r\leq s),\\
& D_{II}(p)=\{Z=(z_{jk}): I-Z\bar{Z^t}>0,\ \text{where}\ Z\ \text{is\ a\ symmetric\ matrix\ of\ order\ }p\},\\
&D_{III}(q)=\{Z=(z_{jk}): I-Z\bar{Z^t}>0,\ \text{where}\ Z\ \text{is\ a\ skew-symmetric\ matrix\ of\ order\ }q\},\\
&D_{IV}(n)=\{Z=(z_1, \cdots , z_n)\in\mathbb{C}^n: 1+|ZZ^t|^2-2Z\bar{Z}^t>0,\ 1-|Z{Z}^t|>0\}.
\end{split}
\end{equation*}

Let $\Omega$ be a classical bounded symmetric domain,
then the Cartan-Hartogs domain $\hat{\Omega}_k$
associated to $\Omega$ is defined to be
\begin{equation}\label{Cartan-Hartogs domain}\hat{\Omega}_k=\{(Z, W)\in\Omega\times\mathbb{C}^m;\parallel W\parallel^2<N(Z,Z)^k\},\end{equation}
where $m$ is a positive integer and $k$ is a positive real number,
$\parallel W\parallel$ is the standard Hermitian norm of $W$,
and the generic norm $N(Z,Z)$ for $D_I{(r,s)}$, $D_{II}(p)$, $D_{III}(q)$, $D_{IV}(n)$
are respectively $\det(I-Z\bar{Z}^t), $ $\det(I-Z\bar{Z}^t)$, $\det(I+Z\bar{Z}^t)$, and $1+|ZZ^t|^2-2Z\bar{Z}^t$.

In \cite{Yin2000},
Yin computed the automorphism groups and Bergman kernels
of Cartan-Hartogs domains explicitly.
Motivated by Liu-Sun-Yau's work \cite{Liu}, Yin proposed the following
open problem: whether Cartan-Hartogs domains are homogeneous regular \cite{Yin2007}? In
this section, we give an affirmative answer to this question.

\begin{thm}\label{thm:Cartan-Hartogs are Homo. Reg.}
Let $\hat{\Omega}_k$ be a Cartan-Hartogs domain defined as above:
\begin{enumerate}
  \item for any $P_0=(Z_0, W_0)\in\partial \hat{\Omega}_k$,
        if $W_0\neq 0$, then
        $$\lim_{P\rightarrow P_0}s_{\hat{\Omega}_k}(P)=1;$$
  \item $\hat{\Omega}_k$ is homogenous regular.
\end{enumerate}
\end{thm}

\begin{rem}
The same estimate as in (1) of Theorem \ref{thm:Cartan-Hartogs are Homo. Reg.} does not hold
for boundary point $P_0=(Z_0, W_0)\in\partial \hat{\Omega}_k$ with $W_0=0$.
In fact, such $P_0$ is an accumulation  boundary point of $\hat{\Omega}_k$.
If $\lim_{P\rightarrow P_0}s_{\hat{\Omega}_k}(P)=1$,
by the same argument as in the proof of Corollary \ref{cor:Wong ball},
one can prove that $s_{\hat{\Omega}_k}\equiv 1$ and hence
$\hat{\Omega}_k$ is biholomorphic to the unit ball,
which is a contradiction.
\end{rem}

Consequently,
by the work of Yeung in \cite{Yeung} and
the results in \S \ref{sec:comparison. intrinsic form},
we have
\begin{cor}
Let $D$ be a Cartan-Hartogs domain,
then
\begin{enumerate}
  \item $D$ is hyperconvex, i.e.,
        $D$ admits a bounded exhaustive plurisubharmonic function;
  \item the Bergman metric and the K\"ahler-Einstein metric
        on $D$ have bounded geometry,
        i.e., the injective radius have positive lower bound
        and the curvatures are bounded;
  \item the Kobayashi metric, the Carath\'eodory metric,
        the Bergman metric, and the K\"ahler-Einstein metric
        on $D$ are equivalent;
  \item all the intrinsic measures considered in \S \ref{sec:comparison. intrinsic form}
        on $D$ are equivalent.
\end{enumerate}
\end{cor}

\begin{proof}(Proof of Theorem \ref{thm:Cartan-Hartogs are Homo. Reg.})
Let $X:\Omega\times\mathbb{C}\rightarrow
\lbrack0,1)$ be defined as
\begin{equation}
X(Z,W)=\frac{\left\Vert W\right\Vert ^{2}}{N(Z,Z)^{k}}-1. \label{EQN7}%
\end{equation}
Then $X$ is a defining function of $\widehat{\Omega}_k$ in $\Omega
\times\mathbb{C}^m$.

We give the proof of the theorem in the case that
$\Omega=D_I(r,s)$ is a bounded symmetric domain
of the first type in the above list.
In this case, $N(Z,Z)=\det(I-Z\bar{Z}^t)$.
Other cases can be proved with the similar argument.

For any point $(Z,W)\in \hat{\Omega}_k$,
there exists an automorphism $f$ of $\hat{\Omega}_k$
such that $f(Z,W)=(0,\cdots,0,a)$ for some $a>0$ (see \cite{Yin2000}).
Assume $\{P_j\}\subset \widehat{\Omega}_k$ with $P_j\rightarrow P_0$ as $j\rightarrow\infty$,
and $\{f_j\}$ are automorphisms of $\widehat{\Omega}_k$ with
$f_j(P_j)=(0,\cdots, 0, a_j)$  ($a_j>0$),
then $\lim_{j\rightarrow \infty} a_j=1$.
By the holomorphic invariance and continuity of
squeezing functions and Theorem \ref{thm:g.s.c text},
it suffices to prove that $(0,\cdots,0,1)$ is a g.s.c boundary point of
$\hat\Omega_k$.
We now compute the real Hessian  $\text{Hess}(X)(0,\cdots,0,1)$ of the defining function $X$ at $(0,\cdots,0,1)$, where $X(Z,W)=\frac{\left\Vert W\right\Vert ^{2}}{N(Z,Z)^{k}}-1$ as above. Let $z_{jk}=x_{jk}+\sqrt{-1}y_{jk}, 1\leq j\leq r, 1\leq k\leq s$, then
$\frac{\partial }{\partial x_{jk}}=\frac{\partial }{\partial z_{jk}}+\frac{\partial }{\partial \bar{z}_{jk}}$ and $\frac{\partial }{\partial y_{jk}}=\sqrt{-1}\left(\frac{\partial }{\partial z_{jk}}-\frac{\partial }{\partial \bar{z}_{jk}}\right)$.
It is clear that
$\frac{\partial N}{\partial z_j}|_{z=0}=\frac{\partial N}{\partial \bar z_j}|_{z=0}$ and
$\frac{\partial^2 N}{\partial z_{jk}\partial z_{lq}}|_{z=0}= \frac{\partial^2 N}{\partial \bar{z}_{jk}\partial \bar{z}_{lq}}|_{z=0}=0$ for all $j,k,l,q$.
Note that $dN(Z,Z)=N(Z,Z)\cdot tr\big((I-Z\bar{Z}^t)^{-1}d(I-Z\bar{Z}^t)\big)$.
Direct calculations show that
\[\left.\frac{\partial^2 N(z,z)}{\partial z_{jk}\partial \bar{z}_{lq}}\right|_{z=0}
=-tr\left(E_{jk}E_{lq}^t\right)=\left\{\begin{array}{cc}
                                                            -1, & j=l,k=q; \\
                                                            0, &  \text{otherwise,}
                                                          \end{array}
 \right.\]
 where $E_{jk}$ denotes a $(r\times s)$-matrix whose components are non vanishing only at the $(j,k)$ position.
 Therefore, we get
\[\text{Hess}(X)(0,\cdots,0,1)
=
\left(
  \begin{array}{cc}
    2kI_{2rs} & 0\\
    0 & 2I_{2m}
  \end{array}
\right).\]
Note also that $\nabla X(0,\cdots,0,1)=2\frac{\partial}{\partial u_m}\neq 0$, where $u_m$ is the real part of $w_m$, hence $(0,\cdots,0,1)$
is a globally strongly convex boundary point of $\hat\Omega_k$. On the other hand, it is clear that $\bar{\hat\Omega}_k\cap\{u_m=1\}=\{(0,\cdots,0,1)\}$, so
$(0,\cdots,0,1)$ is a g.s.c boundary point of $\hat\Omega_k$. This completes the proof of the theorem.
\end{proof}

\begin{rem}
For $k$ tends to 0, the sequence of  domains $\hat\Omega_k$ increases to the product domain $\Omega\times B^m.$
By Theorem \ref{thm: increase limit domain} in the present paper and  Theorem 7.3 and Theorem 7.4 in \cite{DGZ},
we have
$$\lim_{k\rightarrow 0}s_{\hat\Omega_k}(Z,W)=s_{\Omega\times B^m}(Z,W)=(1+c_\Omega)^{-1/2},$$
for all $(Z,W)\in\hat\Omega_k$,
where $c_\Omega=r,$ $ p,$ $ [q/2],$ $ 2$ for $\Omega=D_I(r,s),$ $D_{II}(p,q),$ $D_{III}(q),$ $D_{IV}(n)$ respectively.
\end{rem}

\subsection{Geometry of Strongly pseudoconvex domains}\label{subsec:geo. s.p.c domains}
In this subsection, we prove some results about geometry of strongly pseudoconvex domains.
Those results are well known and play important roles in several complex variables.
But we want to show that they are direct consequences of the results in previous sections.
Throughout this subsection, we always assume that $D$ is a bounded strongly pseudoconvex domain
in $\mathbb{C}^n$ with $C^2$ boundary.

\begin{cor}
The Carath\'eodory metric, the Kobayashi metric, the Bergman metric,
and the K\"ahler-Einstein metric on $D$ are equivalent;
and the Bergman/K\"ahler-Einstein metric on $D$ have bounded geometry
(i.e., the curvature is bounded and the injective radius is bounded by a positive constant).
\end{cor}

\begin{proof}
By Theorem \ref{thm:s.p.c}, $D$ is homogenous regular.
So the corollary follows from Theorem \ref{thm:comp. Kob. and Car. metric}
and Theorem \ref{thm:comp. Beg. KE metrics} in \S \ref{sec:comparison. intrinsic form}
and Theorem 2 in \cite{Yeung}.
\end{proof}

\begin{cor}\label{cor:comp. metric. s.p.c}
Denote the Carath\'eodory metric, the Kobayashi metric, the Bergman metric on $D$
by $\mathcal H^C_D$, $\mathcal H^K_D$, and $\mathcal H^B_D$ respectively.
Then
\begin{enumerate}
\item the metrics admit the following comparisons near the boundary:
$$\lim_{z\rightarrow \partial D}\frac{\mathcal H^K_D(z)}{\mathcal H^C_D(z)}=
\sqrt{n+1}\lim_{z\rightarrow \partial D}\frac{\mathcal H^K_D(z)}{\mathcal H^B_D(z)}=1;$$
\item the sectional curvature of the Bergman metric on $D$ tends to $-\frac{4}{n+1}$ asymptotically near the boundary.
\end{enumerate}
\end{cor}
\begin{proof}
By Theorem \ref{thm:comp. Kob. and Car. metric} and Theorem \ref{thm:s.p.c},
we have $\lim_{z\rightarrow \partial D}\frac{\mathcal H^K_D(z)}{\mathcal H^C_D(z)}=1.$
Let $z_n$ be a sequence in $D$ going to $\partial D$.
By Theorem \ref{thm:s.p.c}, $\lim_{n\rightarrow\infty}s_D(z_n)=1$.
So there exists a sequence of injective holomorphic maps $f_n:D\rightarrow B^n$ with $f_n(z_n)=0$,
$f_n(D)\subset f_{n+1}(D)$ and $\bigcup_{n}f_n(D)=B_n$.
Let $K_n(z, \bar z)$ be the diagonal Bergman kernel of $f_n(D)$
and $K_0$ be the diagonal Bergman kernel of $B^n$.
It is well known that $K_n$ converges to $K_0$ uniformly on compact
subsets of $B_n$ (see e.g. \cite{Pflug}).
Since $K_n$ and $K_0$ can be written as sums of the norms of holomorphic functions,
by Cauchy's inequality, the derivatives of $K_n$ still converges to the derivatives of $K_0$
uniformly on compact subsets of $B^n$.
In particular, the Bergman metrics on $f_n(D)$ at $0$ converges to the Bergman metric
on $B_n$ at 0.
By the holomorphic invariance of the Kobayashi metric and the Bergman metric,
we see $$\lim_{n\rightarrow \infty}\frac{\mathcal H^K_D(z_n)}{\mathcal H^B_D(z_n)}=
\lim_{n\rightarrow \infty}\frac{\mathcal H^K_{f_n(D)}(0)}{\mathcal H^B_{f_n(D)}(0)}=
\frac{\mathcal H^K_{B^n}(0)}{\mathcal H^B_{B^n}(0)}=(n+1)^{-1/2}.$$
Note that the curvature can be expressed in terms of second order derivatives of the Bergman metric on $D$,
and the sectional curvature of the unit ball w.r.t the Bergman metric is $-\frac{4}{n+1}$.
By the holomorphic invariance of Bergman metric,
we see that the sectional curvature of the Bergman metric on $D$ tends to $-\frac{4}{n+1}$ asymptotically near the boundary.
\end{proof}

\begin{cor}
Let $\mathcal{M}$ and $\mathcal{M}'$ be any two of the five measures, i.e.,
the Carath\'eodory measure, the Eisenman-Kobayashi measure,
the measure of the Carath\'eodory metric,
the measure of the Kobayashi metric, and $(n+1)^{-n}$ times of the measure of the Bergman metric on $D$.
Then
$$\lim_{z\rightarrow\partial D}\frac{\mathcal M(z)}{\mathcal M'(z)}=1.$$
\end{cor}
\begin{proof}
This corollary is derived from a combination of Theorem \ref{thm:s.p.c}, Theorem \ref{thm:Comp. of Car. Kob. Vol.},
and Corollary \ref{cor:comp. metric. s.p.c}.
\end{proof}

\section{Further study}\label{sec:further study}
In this section, we propose some directions related to the topics in the present paper for further study.
\subsection{Potential application in algebraic geometry}
Squeezing functions are originally defined on bounded domains,
which are the classical objects of study in several complex variables.
It is natural to consider whether the theory of squeezing functions can be applied to
more general complex manifolds.
Let $X$ be a complex manifold whose universal covering is $\pi:\tilde X\rightarrow X$.
If $\tilde X$ is biholomorphic to a bounded domain, then $s_{\tilde X}$ is defined.
By the holomorphic invariance of squeezing functions,
$s_{\tilde X}$ can be pushed down to a continuous function on $X$.
If $X$ is compact, ${\tilde X}$ is homogenous regular.
This implies a lot of interesting informations about the geometry of $X$,
as shown in \cite{Yeung}.
However, for generic compact complex manifolds,
their universal covering can not be holomorphic equivalent to a bounded domain.
So the application of this approach is much restricted.
On the other hand,
if restricting ourselves to the context of projective manifolds,
we can say more as follows.

By the uniformization theorem,
the universal covering of a  Riemann surface is either
$\mathbb P^1$, $\mathbb C$ or $\Delta$.
In higher dimensions, there is no similar perfect phenomenon.
On the other hand,
based on Bers theory (\cite{Bers68}),
Griffiths showed that any point $z$ in a projective manifold $X$ admits
a Zariski open neighborhood $U$ such that the universal covering
$\tilde U$ of $U$ is biholomorphic to a contractible bounded domain in $\mathbb C^n$ (\cite{Griffiths71}).
Such a neighborhood $U$ of $z$ is called a Griffiths neighborhood of $z$.
We define
$$s_X(z)=\sup_U\{s_{\tilde U}(\tilde z)\},$$
where the supremum is taken over all Griffiths neighborhoods $U$ of $z$,
and $\tilde z\in\tilde U$ is an inverse image of $z$ under the covering map
from $\tilde U$ to $U$.
As $z$ varies on $X$, we get a function $s_X$ on $X$,
which is also called \emph{the squeezing function} of $X$.
It is clear that $s_X$ is a positive function on $X$ which is invariant under holomorphic transformations.
As in the case of bounded domains,
we need to consider two key problems in the context of projective manifolds.
We need to study which algebraic geometric properties of a projective manifold are encoded in
its squeezing function, and develop methods to estimate squeezing functions on projective manifolds.
In the special case that $X$ is a Riemann surface,
it is clear that $s_X\equiv 1$,
which is nothing but the Riemann Mapping Theorem.
If $X$ is homogenous, then $s_X$ is a constant,
which is a holomorphic invariant of $X$.
Even in the case that $X=\mathbb P^n$, the projective space,
it seems nontrivial to determine the exact value of $s_X$.
If $X$ is a ball quotient,
then $s_X=1$.
It is also interesting to consider possible gap phenomenon.
Namely, for each fixed $n>1$ and $0<r<1$, is there a
projective manifold $X$ of dimension $n$ such that $r$
is the exact lower bound of $s_X$?

\subsection{Holomorphic transformations of strongly pseudoconvex domains.}
Let $D$ be a strongly pseudoconvex bounded domain in $\mathbb C^n$ and $p\in\partial D$.
By definition, a \emph{peak function} on $D$ at $p$ is a holomorphic function $h$ on $\bar D$
(i.e., $h$ is holomorphic on some neighborhood of $\bar D$) such that $h(p)=1$ and
$|h(z)|< 1, z\in\bar D-\{p\}.$
Given Forn{\ae}ss and Wold's result (Theorem \ref{thm:h.g.s.c}),
one can easily recover the well known result that each boundary point of $D$ admits a peak function.
We want to consider the theory of peak functions and globally strongly convexity more carefully.
In each side, there have three levels:
\begin{itemize}
\item \emph{Level One: pointwise existence}.

\emph{Assumption}: Let $D\subset\mathbb C^n$ be a strongly pseudoconvex domain with $C^k (k\geq 2)$ boundary.
Assume $p$ is an arbitrary boundary point of $D$.

\emph{Purpose}: to prove the existence of
\begin{enumerate}
\item a peak function on $D$ at $p$; and
\item a holomorphic injective map $f_p:\bar D\rightarrow\mathbb C^n$ such that
$f_p(p)$ is a g.s.c boundary point of $f_p(D)$.
\end{enumerate}
\item \emph{Level Two: variation with respect to boundary points}.

\emph{Assumption}: Let $D\subset\mathbb C^n$ be a strongly pseudoconvex domain with $C^k (k\geq 2)$ boundary.

\emph{Purpose}:
to prove the existence of
\begin{enumerate}
\item a $C^{k-2}$ map $H:\partial D\times \bar D\rightarrow \mathbb C$
such that, for each $p\in\partial D$, $h_p:=H(p,\cdot):\bar D\rightarrow\mathbb C$
is a peak function on $D$ at $p$; and
\item a $C^{k-2}$ map $F:\partial D\times \bar D\rightarrow \mathbb C^n$
such that, for each $p\in\partial D$, $f_p:=F(p,\cdot):\bar D\rightarrow\mathbb C^n$
is holomorphic and injective such that
$f_p(p)$ is a g.s.c boundary point of $f_p(D)$.
\end{enumerate}

\item \emph{Level Three: variation over a family}.

\emph{Assumption}:
Denote by $\Delta$ the unit disc.
Let $\rho:\Delta\times\mathbb C^n\rightarrow\mathbb R$ be a $C^k$-smooth p.s.h function ($k\geq 2$)
whose restriction $\rho_t$ on each fiber $\{t\}\times\mathbb C^n$
is strictly plurisubharmonic.
Let $\mathcal D=\{(t,z)\in\Delta\times\mathbb C^n; \rho(t,z)<0\}$,
and let $\pi:\mathcal D\rightarrow \Delta$ be the natural projection.
Then $\mathcal D$ can be viewed as a family of strongly pseudoconvex domains
over $\Delta$.
We denote by $D_t$ the fiber over $t$ given by $\mathcal D\cap(\{t\}\times\mathbb C^n)$.
Let $\partial^f\mathcal D=\cup_{t\in\Delta}\partial D_t$,
and $\bar{\mathcal D}^f=\mathcal D\cup \partial^f\mathcal D$.

\emph{Purpose:}
to prove the existence of
\begin{enumerate}
\item a $C^{k-2}$ map $\mathcal H:\partial^f\mathcal D\times_\pi\bar{\mathcal D}^f\rightarrow\mathbb C$
such that for each $p\in\partial D_t\subset\partial^f\mathcal D$,
$h_{p,t}:=\mathcal H(p,\cdot):\bar D_t\rightarrow\mathbb C$ is holomorphic and gives a peak function on $D_t$ at $p$; and
\item a $C^{k-2}$ map $\mathcal F:\partial^f\mathcal D\times_\pi\bar{\mathcal D}^f\rightarrow\mathbb C^n$
such that for each $p\in\partial D_t\subset\partial^f\mathcal D$,
$f_{p,t}:=\mathcal F(p,\cdot):\bar D_t\rightarrow\mathbb C^n$ is holomorphic and injective,
and $f_{p,t}(p)$ is a g.s.c boundary point of $f_{p,t}(D_t)$.

\end{enumerate}
\end{itemize}

As we have seen, Lever One was established by Forn{\ae}ss and Wold.
The existence of $H$ in Level Two is also known,
which plays an important role in the study of geometry of
strongly pseudoconvex domains (see \cite{Graham75}).

\end{document}